\documentclass[11.85pt,reqno]{amsart} 

\usepackage{amsmath,amssymb} %% AMS symbol definitions
\usepackage{amscd}            %% for Commutative Diagrams
\usepackage{amsxtra}          %% defines \spcheck, \spdot, etc.
\usepackage{upref}            %% cross-refs to sections are upright

\usepackage[mathscr]{eucal}

\theoremstyle{plain}

\newtheorem{thm}{Theorem}[section]
\newtheorem{lem}[thm]{Lemma}
\newtheorem{rem}[thm]{Remark}

\newtheorem{definition}[thm]{Definition}
\newtheorem{corollary}[thm]{Corollary}
\newtheorem{exa}[thm]{Example}

\numberwithin{equation}{section}
\newcommand{\be}%
  {\protect\setcounter{equation}{\value{subsubsection}}}  
\newcommand{\ee}%
  {\protect\setcounter{subsubsection}{\value{equation}}}

\begin{document}

\title{HYPERGEOMETRIC FUNCTIONS AND ALGEBRAIC CURVES $y^e=x^d+ax+b$ }


\author{Kewat, Pramod Kumar \&  Kumar, Ram}

\address{Department of Applied Mathematics\\ Indian Institute of Technology (ISM), Dhanbad-826004\\ Jharkhand, India}
\email{pramodk@iitism.ac.in, ramkumarbhu1991@gmail.com}

\begin{abstract}
Let $q$ be a prime power and $\mathbb{F}_q$ be a finite field with $q$ elements. Let $e$ and $d$ be positive integers.
In this paper, for $d\geq2$ and $q\equiv1(\mathrm{mod}~ed(d-1))$,  we calculate the number of points on an algebraic curve 
$E_{e,d}:y^e=x^d+ax+b$ over a finite field $\mathbb{F}_q$ in terms of $_dF_{d-1}$
Gaussian hypergeometric series with multiplicative characters of orders $d$ and $e(d-1)$, and in terms of
$_{d-1}F_{d-2}$ Gaussian hypergeometric series with multiplicative characters of orders $ed(d-1)$ and $e(d-1)$.
This helps us to express the trace of Frobenius endomorphism of an algebraic curve $E_{e,d}$ over a finite field
$\mathbb{F}_q$ in terms of the above hypergeometric series. As applications, we obtain some transformations and special values of  $_2F_{1}$ Gaussian hypergeometric series.

\end{abstract}
\subjclass{11G20, 11T24}
\keywords{Algebraic curves, Gaussian hypergeometric series, Frobenius endomorphisms, Characters of finite fields}



\maketitle

\markboth{P. K. Kewat and R. Kumar}{Hypergeometric functions and algebraic curves}
\section{Introduction}In the $19^{th}$ century Gauss introduced classical hypergeometric series. Since then, many
mathematicians such as Kummer, Ramanujan, Beukers, Stiller and others studied classical hypergeometric series extensively and found many interesting connections between classical hypergeometric series and different mathematical objects.
In 1980's, Greene \cite{greene1987hypergeometric} introduced hypergeometric functions (or Gaussian hypergeometric series)
over finite fields analogous to classical hypergeometric series as finite character sums over  a finite field.
It is found that these functions satisfy many summation and transformation formulas analogous to classical hypergeometric series.
In  a series of papers, many interesting relations have been established between special values of these hypergeometric functions 
and the number of points on certain algebraic curves over finite fields
(see, for examples, \cite{barman2012certain, barman2012hypergeometric, barman2013gautam, Barman13, barmanhyperelliptic,
fuselier2007hypergeometric, koike1995orthogonal, kalita2007, lennon2011, ono1998values}).
\par
Fuselier \cite{fuselier2007hypergeometric} and Lennon \cite{lennon2011} found formulas for the trace of Frobenius endomorphism of a
certain family of elliptic curves in terms of Gaussian hypergeometric series containing characters of order 12. In \cite{Barman13}, Barman and Kalita found the number of solutions of the polynomial equation $x^d+ax+b=0$ over a finite field $\mathbb{F}_q$ in terms of special values of Gaussian hypergeometric series with characters of orders $d$ and $d-1$ under the condition that $q \equiv 1 (\text{mod}~ d(d - 1))$ and $d \geq 2$. The same authors, in \cite{barmanhyperelliptic}, expressed the number of $\mathbb{F}_q$-points on a hyperelliptic curve
in terms of special values of Gaussian hypergeometric series.
\par
Let $e$ and $d$ be positive integers and $E_{e,d}$: $y^e=x^d+ax+b$ be an algebraic curve over a finite field $\mathbb{F}_q$. 
Throughout this paper, we assume that $a,b\neq0$. Let $N_{e,d}$ denotes the number of points on the algebraic curve $E_{e,d}$ over
$\mathbb{F}_q$ excluding points at infinity and $a_q(E_{e,d})$ denotes the trace of Frobenius of the algebraic curve $E_{e,d}$. 
In this paper, for $d \geq 2$ and $q\equiv1(\mathrm{mod}~ed(d-1))$, we express $N_{e,d}$ and $a_q(E_{e,d})$ in terms of $_dF_{d-1}$
and $_{d-1}F_{d-2}$  Gaussian hypergeometric series containing multiplicative characters of orders $d$, $e(d-1)$ and $ed(d-1)$.
We deduce the result of Lennon \cite{lennon2011} on the trace of Frobenius of an elliptic curve from our main results.
In \cite{ono1998values}, Ono obtained special values of hypergeometric functions containing quadratic and trivial characters. Only a few
such values are known for higher order characters. In the last section, we derive some interesting special values of $_2F_1$ hypergeometric
function containing multiplicative characters of order $12$.

\section{Preliminaries}
\label{}
Let $\mathbb F_q$ be a finite field with $q$ elements, where $q=p^n, p$ is a prime number and $n$ is a positive integer.
Note that $\mathbb F^*_q=\mathbb F_q\backslash \{0\}$ is a cyclic multiplicative group of order $q-1.$
A multiplicative character $\chi\colon\mathbb F^*_q\longrightarrow\mathbb C^* $ is a group homomorphism. Throughout, 
we reserve the notations $\varepsilon$ and $\phi$ for the trivial and the quadratic characters,
respectively. Thus, for $x\in\mathbb{F}^*_q$\\
\begin{eqnarray}
\phi(x)=\left(\frac{x}{q}\right)=
 \begin{cases}
1,\mbox{  ~if $x$ is a square of some element in $\mathbb{F}^*_q$,}\\
-1,\mbox{if $x$ is not a square of any element in $\mathbb{F}^*_q$}
\end{cases}
\end{eqnarray}
is the Legendre symbol. The following theorem gives the structure of multiplicative characters on $\mathbb{F}^*_q$.
\begin{thm}\cite{lidl1997finite}
Let $g$ be a generator of the multiplicative group $\mathbb F^*_q$. For $j=0,1,2,\cdots,q-2$, the functions \\
 \begin{eqnarray*}
 \chi_j (g^k)=e^\frac{2{\pi}ijk}{q-1}, ~\text{for}~ k=0,1,2,\cdots,q-2,
 \end{eqnarray*}
 define multiplicative characters on $\mathbb{F}^*_q $.
\end{thm}
The set $ \widehat{\mathbb F_q^*}$ of all multiplicative characters on $\mathbb F_q^*$ is a cyclic group under
multiplication of characters.
One extends the domain of a multiplicative character $\chi$ on $\mathbb F_q^*$ to $\mathbb F_q$ by defining $\chi(0)=0$.\\
Define the additive character $ \theta: \mathbb{F}_{q}\rightarrow \mathbb{C}^* $ by $\theta(\alpha)=\zeta^{tr{(\alpha)}}$,
where $\zeta=e^{\frac{2\pi i}{p}}$ and $tr:\mathbb{F}_{q}\rightarrow \mathbb{F}_{p}$ is the trace map given by \\
\begin{center}
$tr(\alpha)=\alpha +\alpha^{p} +\alpha^{p{^2}} +\cdots +\alpha^{p^{n-1}}.$
\end{center}
Throughout this paper, by capital letters $A, B, C, \ldots$ and Greek letters $\chi, \psi, \ldots$, we will denote multiplicative characters.
Let $\delta$ denote both the function on $\mathbb F_q$ and the function on $ \widehat{\mathbb F_q^*}$:
 \begin{eqnarray}
 \delta(x)=
\begin{cases}
 1 \mbox{\quad if $x=0$};\\
 0 \mbox{\quad if $x\neq 0$},
\end{cases}
\end{eqnarray}
and
\begin{equation}
 \delta(A)=
 \begin{cases}
 1 \mbox{\quad if $A$ is trivial character},\\
 0 \mbox{ \quad otherwise}.
\end{cases}
\end{equation}
Define $\overline{\chi}$ by $\chi\overline{\chi}=\varepsilon$. We write $\sum\limits_{x}$ to denote the sum over all $x$ in 
$\mathbb F_q$ and $\sum\limits_{\chi}$ to denote the sum over all characters of $\mathbb{F}_q$.\\
We recall the definitions of the Jacobi sum
\begin{equation}
 J(A,B)=\sum\limits_{x}{A(x)B(1-x)}
\end{equation}
and the Gauss sum 
\begin{equation}
G(A)=:\sum\limits_{x}{A(x)\zeta^{tr{(x)}}}.
\end{equation}
\begin{definition}\cite{lidl1997finite}
For characters A and B of $\mathbb{F}_q$, the binomial coefficient $\binom{A}{B}$ is defined as\\
\begin{eqnarray*}
\binom{A}{B}=\frac{B(-1)}{q}J(A,\overline{B}).
\end{eqnarray*}
\end{definition}
In terms of the binomial coefficients, $A(1+x)$ can be written as
\begin{eqnarray}
A(1+x)=\delta{(x)}+\frac{q}{q-1}\sum\limits_\chi\binom{A}{\chi}\chi(x).
\end{eqnarray}
Some useful properties of the binomial coefficients which follow easily from properties of the Jacobi sums are
\begin{equation}
\binom{A}{B}=\binom{A}{A\overline{B}}, 
\end{equation}
\begin{equation}
\binom{A}{B}=\binom{B\overline{A}}{B}B(-1) 
\end{equation}
and
\begin{equation}
\binom{A}{B}=\binom{\overline{B}}{\overline{A}}AB(-1). 
\end{equation}
Let $T$ be a fixed generator of $ \widehat{\mathbb F_q^*}$ and $G_{m}=G(T^m).$
The following lemma gives the beautiful properties of the Gauss sum. 
\begin{lem}\cite{fuselier2007hypergeometric}\label{l2.3}
If $T^m$ is not trivial for $m\in\mathbb{N},$ then we have\\
\begin{align*}
G_mG_{-m}=qT^m(-1).
\end{align*}
\end{lem}
The following lemmas give the nice relationship between the Gauss sum and the Jacobi sum.
\begin{lem}\cite{lennon2011}\label{l2.4}
If $T^{m-n}$ is not trivial for $m,n\in\mathbb{N},$ then we have
\begin{align*}
G_mG_{-n}=&q\binom{T^m}{T^n}G_{m-n}T^n(-1)\\
=&J(T^m,T^{-n})G_{m-n}.
\end{align*}
\end{lem}
\begin{lem}\cite{lidl1997finite}\label{GJ}
 If $T^{k_1}, T^{k_2}, \dots, T^{k_n}$ are nontrivial multiplicative characters of $\mathbb{F}_q$ and $T^{k_1+k_2+\cdots+k_n}$ is nontrivial, then
\begin{align*}
J(T^{k_1}, T^{k_2}, \dots, T^{k_n})=\frac{G_{k_1}G_{k_2}\dots G_{k_n}}{G_{k_1+k_2+ \dots+ k_n}}.
\end{align*}
\end{lem}
\begin{lem}\cite{fuselier2007hypergeometric}\label{l2.6}
If $T$ is a fixed generator of $\mathbb{\widehat{F}}_{q}^*$ and $n\in\mathbb{N}$, then orthogonality 
relations for multiplicative characters are given by
\begin{enumerate}
 \item $\sum_{x\in\mathbb{F}_q^*}T^n(x)=$ $\begin{cases} q-1 \quad \mbox{if}~ T^n=\varepsilon;\\
0 \quad \quad\quad\mbox{if}~ T^n\neq\varepsilon. \end{cases}$\\
\item $\sum_{n=0}^{q-2}T^n(x)=$
$\begin{cases}
q-1~ \mbox{if}~ x=1;\\
0\quad\quad\mbox{if}~ x\neq1.
\end{cases}$
\end{enumerate}
\end{lem}
The following lemma gives the relation between additive characters and the Gauss sums.
\begin{lem}\cite{fuselier2007hypergeometric}\label{l2.5}
Let $\theta$ be an additive character and $\alpha\in\mathbb{F}^*_q$. We have 
\begin{align*}
\theta(\alpha)=\frac{1}{q-1}\sum_{m=0}^{q-2}{G_{-m}T^m(\alpha)}. 
\end{align*}
\end{lem}
\begin{thm}[Davenport-Hasse Relation \cite{slangcyclotomicfields2012}] Let $m\in\mathbb{Z}^+$ and
 $q\equiv1(\mathrm{mod}~m).$ For multiplicative characters $\chi,\psi\in\mathbb{\widehat{F}}_q^*$, we have
\begin{align*}
\prod_{\chi^m=\varepsilon}{G(\chi\psi)}=-G(\psi^m)\psi(m^{-m})\prod_{\chi^m=\varepsilon}{G(\chi)}. 
\end{align*}
We have the following two special cases of the above theorem.
\begin{corollary}\cite{Barman13}
Let $d$ be a positive integer, $l\in\mathbb{Z}, q\equiv1(\mathrm{mod~d})$ and $t\in\{1,-1\}$.
If $d>1$ is an odd integer, then
\begin{center}
$G_lG_{l+t\frac{q-1}{d}}G_{l+t\frac{2(q-1)}{d}}\cdots G_{l+t\frac{(d-1)(q-1)}{d}} = q^{\frac{d-1}{2}} T^{\frac{(d-1)(d+1)(q-1)}{8d}}\left(-1 \right) T^{-l}\left(d^d\right) G_{ld}.$
\end{center}
If $d$ is an even integer, then
\begin{center}
 $G_{l}G_{l+t\frac{q-1}{d}}G_{l+t\frac{2(q-1)}{d}}\cdots G_{l+t\frac{(d-1)(q-1)}{d}}=q^{\frac{d-2}{2}}G_{\frac{q-1}{2}}
T^{\frac{(d-2)(q-1)}{8}}(-1)
T^{-l}\left(d^d\right)G_{ld}.$
\end{center}
\end{corollary}
 The following lemma gives the values of the Gauss sum at the trivial and the quadratic characters.
\end{thm}
\begin{lem}\cite{fuselier2007hypergeometric}\label{l2.9}
We have
\begin{enumerate}
 \item $G(\varepsilon)=G_0=-1$,\\
 \item $G(\phi)=G_{\frac{q-1}{2}}=\begin{cases}
\sqrt{q}\quad\mbox{if}~ q\equiv1(\mathrm{mod}~4);\\
i\sqrt{q}\quad\mbox{if}~q\equiv3(\mathrm{mod}~4).
\end{cases}
$
\end{enumerate}
\end{lem}
\begin{thm}\cite{kireland}
Let $\theta$ be an additive character and $x,y,z\in\mathbb{F}_q$. Then we have
\begin{align*}
\sum_{z\in\mathbb{F}_q}{\theta(z(x-y))}=q\delta(x,y), 
\end{align*}
where $\delta(x,y)=\begin{cases}
1\quad\mbox{if}~x=y;\\
0\quad\mbox{if}~x\neq y.
\end{cases}
$
\end{thm}
\begin{definition}\cite{lidl1997finite}\label{d2.11}
If $A_0, A_1,\cdots, A_n$ and $B_1, B_2,\cdots,B_n$ are characters of  $\mathbb{F}_q$ and $x\in\mathbb{F}_q$, then the Gaussian 
hypergeometric series $_{n+1}F_n
$ over $\mathbb{F}_q$ is defined as\\
 \begin{equation*}
 _{n+1}F_n
\left(
\begin{matrix}
A_0,&A_1,&\cdots, A_n&\\
&&&|x\\
&B_1,&\cdots, B_n&
\end{matrix}\right)
=\frac{q}{q-1}\sum\limits_{\chi}{\binom{A_0\chi}{\chi}\binom{A_1\chi}{B_1\chi}\cdots \binom{A_n\chi}{B_n\chi}}\chi(x).
\end{equation*}
\end{definition}
\begin{thm}\cite{greene1987hypergeometric}\label{tg4.4}
 For characters $A,B,C$ on $\mathbb{F}_q$ and $x\in\mathbb{F}_q$,
\begin{enumerate}
 \item 
$
 _{2}F_1
\left(
\begin{matrix}
A,&B&\\
&&|x\\
&C&
\end{matrix}\right)
=A(-1) _{2}F_1
\left(
\begin{matrix}
A,&B&\\
&&|1-x\\
&AB\overline{C}&
\end{matrix}\right)
+A(-1)\binom{B}{\overline{A}C}\delta(1-x)-\binom{B}{C}\delta(x),
$
\item
$
 _{2}F_1
\left(
\begin{matrix}
A,&B&\\
&&|x\\
&C&
\end{matrix}\right)
=C(-1)\overline{A}(1-x) _{2}F_1
\left(
\begin{matrix}
A,&C\overline{B}&\\
&&|\frac{x}{x-1}\\
&C&
\end{matrix}\right)
+A(-1)\binom{B}{\overline{A}C}\delta(1-x).
$
\end{enumerate}
\end{thm}

\section{Main Theorems}
In this section, we state and prove the main theorems of this paper.
\begin{thm}\label{mt1}
Let $e$ and $d$ be positive integers and let $N_{e,d}$ denote the number of $\mathbb{F}_q$ points on $E_{e,d}:y^e=x^d+ax+b$ excluding
 the points at infinity.
If $d\geq2$ is an even integer, and $q\equiv1{(\mathrm{mod}{(ed(d-1))})}$, then
\begin{equation*}
N_{e,d}=q+\sum_{i=1}^{e-1}T^{-\frac{i(q-1)}{e}}(b)+\frac{1}
{T^{\frac{(d-2)(2d-1)(q-1)}{8(d-1)}}(-1)}\sum_{i=1}^{e-1}M_iT^{\frac{i(q-1)}{e}}\left(\frac{d-1}{b}\right)\times
\end{equation*}
\begin{align*}
_{d}F_{d-1}
\left(
\begin{matrix}
\phi,&\varepsilon,&\chi,&\chi^2,\cdots,\chi{\frac{d-2}{2}},&\chi^{\frac{d+2}{2}},\cdots,\chi^{d-1}\\
&&&&&|\alpha\\
&\psi^{(\frac{d}{2}e-i)},&\psi^{e-i},&\psi^{2e-i},\cdots,\psi^{\frac{d-2}{2}e-i},&\psi^{\frac{d+2}{2}e-i},\cdots,\psi^{(d-1)e-i}
\end{matrix}\right),
\end{align*}
where $\alpha=\frac{d}{a}{(\frac{bd}{a(d-1)})^{d-1}}$, $\chi$ and $\psi$ are characters of orders $d$ and $e(d-1)$
respectively and
\begin{align*}
M_i=&G_{-\frac{i(q-1)}{e}}G_{-\frac{(\frac{d}{2}e-i)(q-1)}{e(d-1)}}G_{\frac{(id-e)(q-1)}{ed(d-1)}}
G_{\frac{(id-2e)(q-1)}{ed(d-1)}}\cdots
G_{\frac{(id-(\frac{d}{2}-1)e)(q-1)}{ed(d-1)}}
G_{\frac{id-\left(\frac{d}{2}+1\right)e(q-1)}{ed(d-1)}}\\
&\cdots G_{\frac{(id-(d-1)e)(q-1)}{ed(d-1)}}, i=1,2,\cdots,(e-1).
\end{align*}
\end{thm}
\begin{rem}
In the above theorem, by using Lemmas \ref{l2.4} and \ref{GJ}, we can simplify the expression of $M_i$. If $e = 2$, then 
$ M_1 =q^{\frac{d}{2}}T^{\frac{q-1}{2}}(-1)$, and if $ e \neq 2$, then
\begin{align*}
M_i=&q^2\frac{G_{\frac{(2i-e)(q-1)}{2e(d-1)}}}{G_{\frac{(id-\frac{d}{2}e)(q-1)}{ed(d-1)}}}
\binom{T^{\frac{(2i-e)(q-1)}{2e}}}{T^{\frac{i(q-1)}{e}}}
\binom{T^{\frac{q-1}{2}}}{T^{\frac{(\frac{d}{2}e-i)(q-1)}{e(d-1)}}}T^{\frac{(2i(d-2)+ed)(q-1)}{2e(d-1)}}(-1)\\
&\times J\left(T^{\frac{(id-e)(q-1)}{ed(d-1)}}, T^{\frac{(id-2e)(q-1)}{ed(d-1)}}, \dots, T^{\frac{(id-(d-1)e)(q-1)}{ed(d-1)}}\right),i=1,2,\cdots,(e-1).
\end{align*}
\end{rem}
\begin{thm}\label{mt2}
Let $e$ and $d$ be positive integers and let $N_{e,d}$ denote the number of $\mathbb{F}_q$ points on $E_{e,d}:y^e=x^d+ax+b$ excluding the 
 points at infinity. If $d\geq2$ is an odd integer, and $q\equiv1{(\mathrm{mod}{(ed(d-1))})}$, then
\begin{align*}
&N_{e,d}=\\
&q+\sum_{i=1}^{e-1}T^{-\frac{i(q-1)}{e}}(b)-\frac{T^{\frac{(3d-1)(q-1)}{8d}}(-1)}{q^{(d-1)}G_{\frac{q-1}{2}}}
\sum_{i=1}^{e-1}N_iG_{-\frac{i(q-1)}{e}}T^{\frac{i(q-1)}{e}}(-1)
T^{\frac{(e-i)(q-1)}{e}}\left(\frac{b}{d-1}\right)\\
&+\frac{T^{\frac{(4d^2+3d-1)(q-1)}{8d}}(-1)}{G_{\frac{q-1}{2}}}\sum_{i=1}^{e-1}G_{-\frac{i(q-1)}{e}}
T^{\frac{-i(q-1)}{e}}\left(\frac{b}{d-1}\right)T^{\frac{(e-i)(q-1)}{e(d-1)}}\left(-\frac{1}{\alpha}\right)M_i\times\\
&_{d-1}F_{d-2}
\left(
\begin{matrix}
\eta^{id-e},&\eta^{de+id-2e},\cdots,\eta^{\frac{ed^2-4ed+2id+e}{2}},&\eta^{\frac{ed^2-2ed+2id-e}{2}},\cdots,\eta^{ed^2-3ed+id+e}\\
&&\quad\quad\quad\quad\quad\quad\quad\quad\quad\quad\quad\quad|-\alpha\\
&\psi^e,\cdots,\psi^{\frac{d-3}{2}e},&\psi^{\frac{d-1}{2}e},\cdots,\psi^{(d-2)e}
\end{matrix}\right),
\end{align*}
where $\alpha=\frac{d}{a}{(\frac{bd}{a(d-1)})^{d-1}}$, $\eta$ and $\psi$ are characters of orders $ed(d-1)$ and $e(d-1)$
respectively,\\
$M_i=G_{\frac{(id-e)(q-1)}{ed(d-1)}}$ $G_{\frac{(id-2e)(q-1)}{ed(d-1)}}\cdots$
$G_{\frac{\left(id-\frac{d-1}{2}e\right)(q-1)}{ed(d-1)}}$ $G_{\frac{\left(id-\frac{d+1}{2}e\right)(q-1)}{ed(d-1)}}$
$\cdots G_{\frac{(id-(d-1)e)(q-1)}{ed(d-1)}}$\\
and\\
$N_i=\{G_{\frac{q-1}{d}}G_{-\frac{(e-i)(q-1)}{e(d-1)}}\}
\{G_{\frac{2(q-1)}{d}}G_{-\frac{(2e-i)(q-1)}{e(d-1)}}\}\cdots
\{G_{\frac{(d-1)(q-1)}{2d}}~G_{-\frac{(\frac{(d-1)}{2}e-i)(q-1)}{e(d-1)}}\}\\
~~~~~~\times\{G_{\frac{(d+1)(q-1)}{2d}}G_{-\frac{(\frac{d+1}{2}e-i)(q-1)}{e(d-1)}}\}\cdots
\{G_{\frac{(d-1)(q-1)}{d}}G_{-\frac{((d-1)e-i)(q-1)}{e(d-1)}}\}$,
$i=1,2,\cdots,(e-1)$.
\end{thm}
\begin{rem}
In the above theorem, by using Lemmas \ref{l2.3}, \ref{l2.4} and \ref{GJ}, we can simplify the expressions of $M_i$ and $N_i$.
If $e = 2$, then $M_1=q^{\frac{d-1}{2}}T^{-\frac{(d-1)(q-1)}{8d}}(-1)$, and $N_1=q^{d-1}T^{-\frac{(d-1)(q-1)}{8d}}(-1)$. If $ e \neq 2$, then
\begin{equation*}
M_i=G_{\frac{(2i-e)(q-1)}{2e}}J\left(T^{\frac{(id-e)(q-1)}{ed(d-1)}}, T^{\frac{(id-2e)(q-1)}{ed(d-1)}}, 
\dots, T^{\frac{(id-(d-1)e)(q-1)}{ed(d-1)}} \right), i=1,2,\cdots,(e-1)
\end{equation*}
and
\begin{align*}
N_i=& q^{\frac{d-1}{2}}T^{\frac{(d^2-1)(q-1)}{8d}}(-1) & G_{-\frac{(ed-2i)(q-1)}{2e}}
J\left(T^{-\frac{(e-i)(q-1)}{e(d-1)}}, T^{-\frac{(2e-i)(q-1)}{e(d-1)}}, 
\dots, T^{-\frac{((d-1)e-i)(q-1)}{e(d-1)}}\right),\\
&i=1,2,\cdots, (e-1).
\end{align*}
\end{rem}

%\section{Proof of main Theorems}
\begin{proof}[\bf{Proof of the Theorem \ref{mt1}}]
 Let $P(x,y)=x^d+ax+b-y^e$. Then
 $N_{e,d}=\#\{(x,y)\in{\mathbb{F}_q^2}:P(x,y)=0\}.$ Now using the elementary identity,
\begin{eqnarray*}
\sum\limits_{z\in\mathbb{F}_q}{\theta(zP(x,y))}=
\begin{cases}
q \quad \text{if}~ P(x,y)=0;\\
0 \quad \text{if}~ P(x,y)\neq 0,
\end{cases}
\end{eqnarray*}
we obtain that
\begin{align*}
qN_{e,d}=&\sum\limits_{x,y,z\in\mathbb{F}_q}{\theta(zP(x,y))}\\
=&\sum\limits_{x,y\in\mathbb{F}_q}{\theta(0P(x,y))}+\sum\limits_{z\in\mathbb{F}^*_q}{\theta(zP(0,0))}+
\sum\limits_{y,z\in\mathbb{F}^*_q}{\theta(zP(0,y))}\\
&+\sum\limits_{x,z\in\mathbb{F}^*_q}{\theta(zP(x,0))}+
\sum\limits_{x,y,z\in\mathbb{F}^*_q}{\theta(zP(x,y))}\\
=&q^2+\sum\limits_{z\in\mathbb{F}^*_q}\theta(bz)+\sum\limits_{y,z\in\mathbb{F}^*_q}\theta(bz)\theta(-zy^e)+
\sum\limits_{x,z\in\mathbb{F}^*_q}\theta(bz)\theta(zx^d)\theta(azx)\\
&+\sum\limits_{x,y,z\in\mathbb{F}^*_q}\theta(bz)\theta(zx^d)\theta(azx)\theta(-zy^e).
\end{align*}
This we can write as
\begin{align}
qN_{e,d}:=&q^2+A+B+C+D,\label{e3.1}
\end{align}
where 
\begin{align*}
A&=\sum\limits_{z\in\mathbb{F}^*_q}\theta(bz),~~
B=\sum\limits_{y,z\in\mathbb{F}^*_q}\theta(bz)\theta(-zy^e),\\
C&=\sum\limits_{x,z\in\mathbb{F}^*_q}\theta(bz)\theta(zx^d)\theta(azx) ~~\text{and}~~
D=\sum\limits_{x,y,z\in\mathbb{F}^*_q}\theta(bz)\theta(zx^d)\theta(azx)\theta(-zy^e).
\end{align*}
Following in a similar fashion as in \cite{barmanhyperelliptic} and using Lemmas \ref{l2.6}, \ref{l2.5} and \ref{l2.9}, we have\\
$A=-1$ and $B=1+q\sum_{i=1}^{e-1}T^{-\frac{i(q-1)}{e}}(b)$.\\
Similarly, we calculate $C$ and $D$,
\begin{align*}
C=\frac{1}{(q-1)^3}\sum_{l,m,n=0}^{q-2}{G_{-l}G_{-m}G_{-n}T^l(b)T^n(a)}\sum\limits_{z\in\mathbb{F}_q^*}{T^{l+m+n}}(z)
\sum\limits_{x\in\mathbb{F}_q^*}{T^{md+n}}(x)
\end{align*}
and
\begin{align*}
D=&\frac{1}{(q-1)^4}\sum_{l,m,n,k=0}^{q-2}{G_{-l}G_{-m}G_{-n}G_{-k}T^l(b)T^n(a)T^k(-1)}\sum\limits_{z\in\mathbb{F}_q^*}
{T^{l+m+n+k}}(z)\\
&\times\sum\limits_{x\in\mathbb{F}_q^*}{T^{md+n}}(x)
\sum\limits_{y\in\mathbb{F}_q^*}{T^{ek}(y)}.
\end{align*}
The innermost sum of $D$ is non zero only if $k=\frac{i(q-1)}{e}, i=0,1,\cdots(e-1).$
For $k=0, D=-C$. Thus, we can write
\begin{align*}
D=&-C+\frac{1}{(q-1)^3}\sum_{i=1}^{e-1}\sum_{l,m,n=0}^{q-2}{G_{-l}G_{-m}G_{-n}G_{-\frac{i(q-1)}{e}}T^l(b)T^n(a)
T^{\frac{i(q-1)}{e}}(-1)}\\
&\times\sum\limits_{z\in\mathbb{F}_q^*}{T^{l+m+n+\frac{i(q-1)}{e}}}(z)
\sum\limits_{x\in\mathbb{F}_q^*}{T^{md+n}}(x).
\end{align*}
Now we can write $D=-C+D^{'}$, where
\begin{align*}
D^{'}=&\frac{1}{(q-1)^3}\sum_{i=1}^{e-1}\sum_{l,m,n=0}^{q-2}{G_{-l}G_{-m}G_{-n}G_{-\frac{i(q-1)}{e}}T^l(b)T^n(a)
T^{\frac{i(q-1)}{e}}(-1)}\\
&\times\sum\limits_{z\in\mathbb{F}_q^*}{T^{l+m+n+\frac{i(q-1)}{e}}}(z)
\sum\limits_{x\in\mathbb{F}_q^*}{T^{md+n}}(x).
\end{align*}
Here, the term $D^{'}=0$ unless $n=-md$ and $l=(d-1)(m+\frac{(e-i)(q-1)}{e(d-1)})$. Thus
\begin{align}
D^{'}=&\frac{1}{(q-1)}\sum_{i=1}^{e-1}G_{-\frac{i(q-1)}{e}}T^{\frac{i(q-1)}{e}}(-1)
\sum_{m=0}^{q-2}{G_{(d-1)\left(-m-\frac{(e-i)(q-1)}{e(d-1)}\right)}}G_{-m}G_{md}\nonumber\\
%  \end{align*}
%  \begin{align}
&\times T^{(d-1)\left(m+\frac{(e-i)(q-1)}{e(d-1)}\right)}(b)
T^{-md}(a).\label{e3.2}
\end{align}
By putting the values of $A,B,C$ and $D$ in \eqref{e3.1}, we get
\begin{equation}
qN_{e,d}=q^2+q\sum_{i=1}^{e-1}T^{-\frac{i(q-1)}{e}}(b)+D^{'}.\label{e3.3}
\end{equation}
If $d\geq2$ is an even integer and $m\in\mathbb{Z}$, then Davenport-Hasse relation for $q\equiv1{(\mathrm{mod}~d)},
~t\in\{1,-1\}$ gives
\begin{equation*}
G_{dm}=\frac{G_mG_{m+\frac{q-1}{d}}G_{m+\frac{2(q-1)}{d}}\cdots G_{m+\frac{(d-1)(q-1)}{d}}}{q^{\frac{d-2}{2}}G_{\frac{q-1}{2}}
T^{\frac{(d-2)(q-1)}{8}}(-1)T^{-m}(d^d)}
\end{equation*}
and
\begin{equation*}
G_{(d-1)\big(-m-\frac{(e-i)(q-1)}{e(d-1)}\big)}=\frac{G_{-m-\frac{(e-i)(q-1)}{e(d-1)}}G_{-m-\frac{(2e-i)(q-1)}{e(d-1)}}
\cdots G_{-m-\frac{((d-1)e-i)(q-1)}{e(d-1)}}}{q^{\frac{d-2}{2}}T^{\frac{d(d-2)(q-1)}{8(d-1)}}(-1)
T^{m+\frac{(e-i)(q-1)}{e(d-1)}}(d-1)^{d-1}},\\
i=1,2,\cdots,(e-1).\\
\end{equation*}
Now using $G_{md}$ and $G_{(d-1)\big(-m-\frac{(e-i)(q-1)}{e(d-1)}\big)}$ in \eqref{e3.2}, we get\\
\begin{align*}
D^{'}=&\frac{1}
{q^{d-2}(q-1)G_{\frac{q-1}{2}}T^{\frac{(d-2)(2d-1)(q-1)}{8(d-1)}}(-1)}
\sum_{i=1}^{e-1}\frac{G_{-\frac{i(q-1)}{e}}T^{\frac{(e-i)(q-1)}{e}}(b)T^{\frac{i(q-1)}{e}}(-1)}{T^{\frac{(e-i)(q-1)}{e}}(d-1)}\\
&\times\sum_{m=0}^{q-2}\{G_{m}G_{-m}\}
\{G_{m+\frac{q-1}{d}}G_{-m-\frac{(e-i)(q-1)}{e(d-1)}}\}
\{G_{m+\frac{2(q-1)}{d}}G_{-m-\frac{(2e-i)(q-1)}{e(d-1)}}\}\cdots\\&\times\{G_{m+\frac{(d-2)(q-1)}{2d}}
G_{-m-\frac{(\frac{d-2}{2} e-i)(q-1)}{e(d-1)}}\}
\{G_{m+\frac{q-1}{2}}G_{-m-\frac{(\frac{d}{2}e-i)(q-1)}{e(d-1)}}\}\\
&\times\{G_{m+\frac{(d+2)(q-1)}{2d}}G_{-m-\frac{(\frac{d+2}{2}e-i)(q-1)}{e(d-1)}}\}\cdots
\{G_{m+\frac{(d-1)(q-1)}{d}}G_{-m-\frac{((d-1)e-i)(q-1)}{e(d-1)}}\}\\
&\times T^m\left(\frac{b^{d-1}d^d}{(d-1)^{d-1}a^d}\right).
\end{align*}
By arranging these terms, we get 
\begin{align*}
D^{'}=&\frac{1}
{q^{d-2}(q-1)G_{\frac{q-1}{2}}T^{\frac{(d-2)(2d-1)(q-1)}{8(d-1)}}(-1)}
\sum_{i=1}^{e-1}\frac{G_{-\frac{i(q-1)}{e}}T^{\frac{(e-i)(q-1)}{e}}(b)T^{\frac{i(q-1)}{e}}(-1)}{T^{\frac{(e-i)(q-1)}{e}}(d-1)}\\
&\times\sum_{m=0}^{q-2}{\{G_{m+\frac{q-1}{2}}G_{-m}\}}
\{G_mG_{-m-\frac{(\frac{d}{2}e-i)(q-1)}{e(d-1)}}\}
\{G_{m+\frac{q-1}{d}}G_{-m-\frac{(e-i)(q-1)}{e(d-1)}}\}\\
&\times\{G_{m+\frac{2(q-1)}{d}}G_{-m-\frac{(2e-i)(q-1)}{e(d-1)}}\}
\cdots\{ G_{m+\frac{(d-2)(q-1)}{2d}} G_{-m-\frac{(\frac{d-2}{2}e-i)(q-1)}{e(d-1)}}\}\\
&\times\{ G_{m+\frac{(d+2)(q-1)}{2d}} G_{-m-\frac{(\frac{d+2}{2}e-i)(q-1)}{e(d-1)}}\}\cdots
 \{ G_{m+\frac{(d-1)(q-1)}{d}} G_{-m-\frac{(d-1)e-i)(q-1)}{e(d-1)}}\}\\
 &\times T^m\left(\frac{b^{d-1}d^d}{(d-1)^{d-1}a^d}\right).
\end{align*}
By using Lemma \ref{l2.4} in the above equation, we have 
\begin{align*}
D^{'}=&\frac{q^2}
{(q-1)G_{\frac{q-1}{2}}T^{\frac{(d-2)(2d-1)(q-1)}{8(d-1)}}(-1)}
\sum_{i=1}^{e-1}G_{-\frac{i(q-1)}{e}}T^{\frac{i(q-1)}{e}}(-\frac{d-1}{b})\sum_{m=0}^{q-2}{G_{\frac{q-1}{2}}}T^m(-1)\\
&\times {\binom{T^{m+\frac{q-1}{2}}}{T^m}}
G_{-\frac{(\frac{d}{2}e-i)(q-1)}{e(d-1)}}
T^{m+\frac{(\frac{d}{2}e-i)(q-1)}{e(d-1)}}(-1)\binom{T^{m}}{T^{m+\frac{(\frac{d}{2}e-i)(q-1)}{e(d-1)}}}
G_{\frac{(id-e)(q-1)}{ed(d-1)}}\\
&\times T^{m+\frac{(e-i)(q-1)}{e(d-1)}}(-1)\binom{T^{m+\frac{q-1}{d}}}{T^{m+\frac{(e-i)(q-1)}{e(d-1)}}}
G_{\frac{(id-2e)(q-1)}{ed(d-1)}}
T^{m+\frac{(2e-i)(q-1)}{e(d-1)}}(-1)
\binom{T^{m+\frac{2(q-1)}{d}}}{T^{m+\frac{(2e-i)(q-1)}{e(d-1)}}}\\
&\times\cdots G_{\frac{(id-(\frac{d}{2}-1)e)(q-1)}{ed(d-1)}}
T^{m+\frac{(\frac{d-2}{2}e-i)(q-1)}{e(d-1)}}(-1)
 \binom{T^{m+\frac{(d-2)(q-1)}{2d}}}{T^{m+\frac{(\frac{id-2}{2}e-1)(q-1)}{e(d-1)}}}G_{\frac{(id-(\frac{d}{2}+1)e)(q-1)}{ed(d-1)}}\\
 &\times T^{m+\frac{(\frac{d+2}{2}e-i)(q-1)}{e(d-1)}}(-1)
\binom{T^{m+\frac{(d+2)(q-1)}{2d}}}{T^{m+\frac{(\frac{d+2}{2}e-i)(q-1)}{e(d-1)}}}
\cdots G_{\frac{(id-(d-1)e)(q-1)}{ed(d-1)}}T^{m+\frac{((d-1)e-i)(q-1)}{e(d-1)}}(-1)\\
&\times \binom{T^{m+\frac{(d-1)(q-1)}{d}}}{T^{m+\frac{((d-1)e-i)(q-1)}{e(d-1)}}}
T^m\left(\frac{b^{(d-1)}d^d}{(d-1)^{(d-1)}a^d}\right).
\end{align*}
By collecting the terms of $G_s$ and $T^s$, we get
\begin{align*}
D^{'}=&\frac{q^2}
{(q-1)T^{\frac{(d-2)(2d-1)(q-1)}{8(d-1)}}(-1)}
\sum_{i=1}^{e-1}M_iT^{\frac{i(q-1)}{e}}\left(-\frac{d-1}{b}\right)\sum_{m=0}^{q-2}T^{md+\frac{d(q-1)}{2}}(-1)\\
&\times T^{-\frac{i(q-1)}{e}}(-1)\binom{T^{m+\frac{q-1}{2}}}{T^m}\binom{T^{m}}{T^{m+\frac{(\frac{d}{2}e-i)(q-1)}{e(d-1)}}}
\binom{T^{m+\frac{q-1}{d}}}{T^{m+\frac{(e-i)(q-1)}{e(d-1)}}}
\binom{T^{m+\frac{2(q-1)}{d}}}{T^{m+\frac{(2e-i)(q-1)}{e(d-1)}}}\\
&\times\cdots\binom{T^{m+\frac{(d-2)(q-1)}{2d}}}
{T^{m+\frac{\frac{d-2}{2}e-i)(q-1)}{e(d-1)}}}
\binom{T^{m+\frac{(d+2)(q-1)}{2d}}}{T^{m+\frac{\frac{d+2}{2}e-i)(q-1)}{e(d-1)}}}\cdots
\binom{T^{m+\frac{(d-1)(q-1)}{d}}}{T^{m+\frac{((d-1)e-i)(q-1)}{e(d-1)}}}
T^m\left(\frac{b^{d-1}d^d}{(d-1)^{d-1}a^d}\right),
\end{align*}
where
\begin{align*}
M_i=&G_{-\frac{i(q-1)}{e}}G_{-\frac{(\frac{d}{2}e-i)(q-1)}{e(d-1)}}G_{\frac{(id-e)(q-1)}{ed(d-1)}}
G_{\frac{(id-2e)(q-1)}{ed(d-1)}}\cdots
G_{\frac{(id-(\frac{d}{2}-1)e)(q-1)}{ed(d-1)}}
G_{\frac{id-\left(\frac{d}{2}+1\right)e(q-1)}{e(d-1)}}\\
&\cdots G_{\frac{(id-(d-1)e)(q-1)}{ed(d-1)}}, i=1,2,\cdots,(e-1).
\end{align*}
Since $d$ is even, thus $T^{md+\frac{d(q-1)}{2}}(-1)=1$. By Definition \ref{d2.11}, we have
\begin{align*}
D^{'}=&\frac{q}
{T^{\frac{(d-2)(2d-1)(q-1)}{8(d-1)}}(-1)}\sum_{i=1}^{e-1}M_iT^{\frac{i(q-1)}{e}}\left(\frac{d-1}{b}\right)\\
% \end{align*}
&\times_{d}F_{d-1}
\left(
\begin{matrix}
\phi,&\varepsilon,&\chi,&\chi^2,\cdots,\chi{\frac{d-2}{2}},&\chi^{\frac{d+2}{2}},\cdots,\chi^{d-1}\\
&&&&&|\alpha\\
&\psi^{(\frac{d}{2}e-i)},&\psi^{e-i},&\psi^{2e-i},\cdots,\psi^{\frac{d-2}{2}e-i},&\psi^{\frac{d+2}{2}e-i},\cdots,\psi^{(d-1)e-i}
\end{matrix}\right),
\end{align*}
where $\alpha=\frac{d}{a}(\frac{bd}{a(d-1)})^{(d-1)},$ $\chi$ and $\psi$ are characters of orders $d$ and $e(d-1)$ respectively.
We complete the proof by putting the value of $D^{'}$ in \eqref{e3.3}.
\end{proof}
\begin{proof}[\bf{Proof of the Theorem \ref{mt2}}]
If $d\geq2$ and is an odd integer, then Davenport-Hasse relations for $G_{dm}$ and $G_{(d-1)(-m-\frac{(e-1)(q-1)}{e(d-1)})}$
 are given by
\begin{align*}
G_{dm}=\frac{G_mG_{m+\frac{q-1}{d}}G_{m+\frac{2(q-1)}{d}}\cdots G_{m+\frac{(d-1)(q-1)}{d}}}{q^{\frac{d-1}{2}}T^{\frac{(d-1)
(d+1)(q-1)}{8d}}(-1)T^{-m}(d^d)},
\end{align*}
\begin{align*}
G_{(d-1)(-m-\frac{(e-i)(q-1)}{e(d-1)})}=&
\frac{G_{-m-\frac{(e-i)(q-1)}{e(d-1)}}G_{-m-\frac{(2e-i)(q-1)}{e(d-1)}}\cdots G_{-m-\frac{((d-1)e-i)(q-1)}{e(d-1)}}}
{G_{\frac{q-1}{2}}q^{\frac{d-3}{2}}T^{\frac{(d-3)(q-1)}{8}}(-1)T^{m+\frac{(e-i)(q-1)}{e(d-1)}}(d-1)^{(d-1)}},\\
&i=1,2,\cdots,(e-1).
\end{align*}
Using these identities in \eqref{e3.2}, we get
\begin{align*}
D^{'}=&\frac{T^{\frac{(3d-1)(q-1)}{8d}}(-1)}{(q-1)q^{(d-2)}G_{\frac{q-1}{2}}
T^{\frac{(d^2-1)(q-1)}{4d}}(-1)}\sum_{i=1}^{e-1}G_{-\frac{i(q-1)}{e}}T^{\frac{i(q-1)}{e}}(-1)
T^{\frac{(e-i)(q-1)}{e}}\left(\frac{b}{d-1}\right)\\
&\times\sum_{m=0}^{q-2}{\{G_mG_{-m}\}}\{{G_{m+\frac{q-1}{d}}G_{-m-\frac{(e-i)(q-1)}{e(d-1)}}\}}
\{G_{m+\frac{2(q-1)}{d}}G_{-m-\frac{(2e-i)(q-1)}{e(d-1)}}\}\cdots\\
&\times\{G_{m+\frac{(d-1)(q-1)}{2d}}~G_{-m-\frac{(\frac{(d-1)}{2}e-i)(q-1)}{e(d-1)}}\}
\{G_{m+\frac{(d+1)(q-1)}{2d}}G_{-m-\frac{(\frac{d+1}{2}e-i)(q-1)}{e(d-1)}}\}\\
&\times\cdots\{G_{m+\frac{(d-1)(q-1)}{d}}G_{-m-\frac{((d-1)e-i)(q-1)}{e(d-1)}}\}
T^m\left(\frac{b^{(d-1)}d^d}{(d-1)^{(d-1)}a^d}\right).
\end{align*}
Since $d$ is an odd integer, thus $T^{\frac{(d^2-1)(q-1)}{4d}}(-1)=1$.
Next, we eliminate the term $\{G_mG_{-m}\}$, using the fact that if $m=0$, then $G_mG_{-m}=qT^m(-1)-(q-1)$, and if $m\neq0$, 
then $G_mG_{-m}=qT^m(-1)$. Using these identities in the above equation and rearranging the second term, we have
\begin{align*}
 D^{'}
=&\frac{qT^{\frac{(3d-1)(q-1)}{8d}}(-1)}{(q-1)q^{(d-2)}G_{\frac{q-1}{2}}
}\sum_{i=1}^{e-1}G_{-\frac{i(q-1)}{e}}T^{\frac{i(q-1)}{e}}(-1)
T^{\frac{(e-i)(q-1)}{e}}\left(\frac{b}{d-1}\right)\\
&\times\sum_{m=0}^{q-2}{\{G_{m+\frac{q-1}{d}}G_{-m-\frac{(e-i)(q-1)}{e(d-1)}}\}}
\{G_{m+\frac{2(q-1)}{d}}G_{-m-\frac{(2e-i)(q-1)}{e(d-1)}}\}\cdots\\
&\{G_{m+\frac{(d-1)(q-1)}{2d}}~G_{-m-\frac{\left(\frac{d-1}{2}e-i\right)(q-1)}{e(d-1)}}\}
\{G_{m+\frac{(d+1)(q-1)}{2d}}G_{-m-\frac{\left(\frac{d+1}{2}e-i\right)(q-1)}{e(d-1)}}\}\\
&\times\cdots\{G_{m+\frac{(d-1)(q-1)}{d}}G_{-m-\frac{((d-1)e-i)(q-1)}{e(d-1)}}\}
T^m\left(-\frac{b^{(d-1)}d^d}{(d-1)^{(d-1)}a^d}\right)\\
&-\frac{T^{\frac{(3d-1)(q-1)}{8d}}(-1)}{q^{(d-2)}G_{\frac{q-1}{2}}
}\sum_{i=1}^{e-1}G_{-\frac{i(q-1)}{e}}T^{\frac{i(q-1)}{e}}(-1)
T^{\frac{(e-i)(q-1)}{e}}\left(\frac{b}{d-1}\right)\\
&\times\{G_{\frac{q-1}{d}}G_{-\frac{(e-i)(q-1)}{e(d-1)}}\}
\{G_{\frac{2(q-1)}{d}}G_{-\frac{(2e-i)(q-1)}{e(d-1)}}\}\cdots
\{G_{\frac{(d-1)(q-1)}{2d}}~G_{-\frac{\left(\frac{d-1}{2}e-i\right)(q-1)}{e(d-1)}}\}
\\&\times\{G_{\frac{(d+1)(q-1)}{2d}}G_{-\frac{\left(\frac{d+1}{2}e-i\right)(q-1)}{e(d-1)}}\}\cdots
\{G_{\frac{(d-1)(q-1)}{d}}G_{-\frac{((d-1)e-i)(q-1)}{e(d-1)}}\}.
\end{align*}
By using Lemma \ref{l2.4} in each terms of the above equation and collecting the terms of $G_s$ and $T^s$ and $T^{m(d-1)}(-1)=1$ (since $d$ is an odd 
integer), we have
\begin{align*}
D^{'}
=&\frac{q^2T^{\frac{(4d^2+3d-1)(q-1)}{8d}}(-1)}{(q-1)G_{\frac{q-1}{2}}
}\sum_{i=1}^{e-1}G_{-\frac{i(q-1)}{e}}T^{\frac{i(q-1)}{e}}(-1)T^{-\frac{i(q-1)}{e}}(-1)
T^{\frac{-i(q-1)}{e}}\left(\frac{b}{d-1}\right)\\
&\times\{G_{\frac{(id-e)(q-1)}{ed(d-1)}}
G_{\frac{(id-2e)(q-1)}{ed(d-1)}}
G_{\frac{(id-3e)(q-1)}{ed(d-1)}}\cdots
G_{\frac{\left(id-\frac{(d-1)}{2}e \right)(q-1)}{ed(d-1)}}
G_{\frac{\left(id-\frac{(d+1)}{2}e\right)(q-1)}{ed(d-1)}}\\
&\times\cdots G_{\frac{(id-(d-1)e)(q-1)}{ed(d-1)}}\}
\sum_{m=0}^{q-2}{\binom{T^{m+\frac{q-1}{d}}}{T^{m+\frac{(e-i)(q-1)}{e(d-1)}}}}
\binom{T^{m+\frac{2(q-1)}{d}}}{T^{m+\frac{(2e-i)(q-1)}{e(d-1)}}}\cdots\\
 &\times\binom{T^{m+\frac{(d-1)(q-1)}{2d}}}{T^{m+\frac{\left(\frac{d-1}{2}e-i\right)(q-1)}{e(d-1)}}}
\binom{T^{m+\frac{(d+1)(q-1)}{2d}}}{T^{m+\frac{\left(\frac{d+1}{2}e-i\right)(q-1)}{e(d-1)}}}
\cdots\binom{T^{m+\frac{(d-1)(q-1)}{d}}}{T^{m+\frac{((d-1)e-i)(q-1)}{e(d-1)}}}
T^m(-\alpha)\\
&-\frac{T^{\frac{(3d-1)(q-1)}{8d}}(-1)}{q^{(d-2)}G_{\frac{q-1}{2}}}\sum_{i=1}^{e-1}G_{-\frac{i(q-1)}{e}}T^{\frac{i(q-1)}{e}}(-1)
T^{\frac{-i(q-1)}{e}}\left(\frac{b}{d-1}\right)N_i,
\end{align*}
where
\begin{align*}
N_i=&\{G_{\frac{q-1}{d}}G_{-\frac{(e-i)(q-1)}{e(d-1)}}\}
\{G_{\frac{2(q-1)}{d}}G_{-\frac{(2e-i)(q-1)}{e(d-1)}}\}\cdots
\{G_{\frac{(d-1)(q-1)}{2d}}~G_{-\frac{\left(\frac{d-1}{2}e-i\right)(q-1)}{e(d-1)}}\}\\
&\times\{G_{\frac{(d+1)(q-1)}{2d}}G_{-\frac{\left(\frac{d+1}{2}e-i\right)(q-1)}{e(d-1)}}\}\cdots
\{G_{\frac{(d-1)(q-1)}{d}}G_{-\frac{((d-1)e-i)(q-1)}{e(d-1)}}\},\\
&i=1,2,\cdots,(e-1),
\end{align*}
and $\alpha=\left(\frac{d}{a}\right)\left(\frac{bd}{a(d-1)}\right)^{(d-1)}$.
Replace $m+\frac{(e-i)(q-1)}{e(d-1)}$ by $m$ in the above equation, we have
\begin{align*}
D^{'}=&q^2\frac{T^{\frac{(4d^2+3d-1)(q-1)}{8d}}(-1)}{(q-1)G_{\frac{q-1}{2}}}\sum_{i=1}^{e-1}G_{-\frac{i(q-1)}{e}}
T^{\frac{-i(q-1)}{e}}\left(\frac{b}{d-1}\right)M_i\sum_{m=0}^{q-2}\binom{T^{m+\frac{(id-e)(q-1)}{ed(d-1)}}}{T^m}\\
&\times\binom{T^{m+\frac{(id+(d-2)e)(q-1)}{ed(d-1)}}}{T^{m+\frac{e(q-1)}{e(d-1)}}}\cdots
\binom{T^{m+\frac{(ed^2-4ed+2id+e)(q-1)}{2ed(d-1)}}}{T^{m+\frac{(d-3)e(q-1)}{2e(d-1)}}}
\binom{T^{m+\frac{(ed^2-2ed+2id-e)(q-1)}{2ed(d-1)}}}{T^{m+\frac{(d-1)e(q-1)}{2e(d-1)}}}\\
&\times\cdots\binom{T^{m+\frac{(ed^2-3ed+id+e)(q-1)}{ed(d-1)}}}{T^{m+\frac{(d-2)e(q-1)}{e(d-1)}}}
T^m(-\alpha)T^{-\frac{(e-i)(q-1)}{e(d-1)}}(-\alpha)\\
&-\frac{T^{\frac{(3d-1)(q-1)}{8d}}(-1)}{q^{(d-2)}G_{\frac{q-1}{2}}}\sum_{i=1}^{e-1}G_{-\frac{i(q-1)}{e}}T^{\frac{i(q-1)}{e}}(-1)
T^{\frac{-i(q-1)}{e}}\left(\frac{b}{d-1}\right)N_i,
\end{align*}
where $M_i=\{G_{\frac{(id-e)(q-1)}{ed(d-1)}}G_{\frac{(id-2e)(q-1)}{ed(d-1)}}
G_{\frac{(id-3e)(q-1)}{ed(d-1)}}\cdots
G_{\frac{\left(id-\frac{d-1}{2}e\right)(q-1)}{ed(d-1)}}G_{\frac{\left(id-\frac{d+1}{2}e \right)(q-1)}{ed(d-1)}}\\
\times G_{\frac{\left(id-\frac{d+3}{2}e\right)(q-1)}{ed(d-1)}}\cdots G_{\frac{(id-(d-1)e)(q-1)}{ed(d-1)}}\}, i=1,2,\cdots,(e-1).$\\ 
By using Definition \ref{d2.11} in the above equation, we have
\begin{align*}
D^{'}=&-\frac{T^{\frac{(3d-1)(q-1)}{8d}}(-1)}{q^{(d-2)}G_{\frac{q-1}{2}}}\sum_{i=1}^{e-1}N_i
G_{-\frac{i(q-1)}{e}}T^{\frac{i(q-1)}{e}}(-1)
T^{\frac{-i(q-1)}{e}}\left(\frac{b}{d-1}\right)\\
&+q\frac{T^{\frac{(4d^2+3d-1)(q-1)}{8d}}(-1)}{G_{\frac{q-1}{2}}}\sum_{i=1}^{e-1}G_{-\frac{i(q-1)}{e}}
T^{\frac{-i(q-1)}{e}}\left(\frac{b}{d-1}\right)M_iT^{-\frac{(e-i)(q-1)}{e(d-1)}}(-\alpha)
\end{align*}
\begin{align*}
\times_{d-1}F_{d-2}
\left(
\begin{matrix}
\eta^{id-e},&\eta^{id+ed-2e},\cdots,\eta^{\frac{ed^2-4ed+2id+e}{2}},&\eta^{\frac{ed^2-2ed+2id-e}{2}},\cdots,\eta^{ed^2-3ed+id+e}\\
&&&|-\alpha\\
&\psi^{e},\cdots,\psi^{\frac{d-3}{2}e},&\psi^{\frac{d-1}{2}e},\cdots,\psi^{(d-2)e}
\end{matrix}\right),
\end{align*}
where $\eta$, $\psi$ are characters of orders 
$ed(d-1)$ and $e(d-1)$ respectively.
We complete the proof by putting the value of $D^{'}$ in equation \eqref{e3.3}. Note that if $e=2$, then 
$N_1$ becomes $q^{d-1}T^{-\frac{(d-1)(q-1)}{8d}}(-1)$ and $M_1=q^{\frac{d-1}{2}}T^{-\frac{(d-1)(q-1)}{8d}(-1)}$.
\end{proof}
\begin{rem}
Let $\pi\in\mathbb{F}_p^*$ be of the order $e$.
If $e=d$ and $p\equiv1~(\mathrm{mod}~e)$, then there are $e$ points at infinity, namely
$[1:1:0],[1:\pi:0],[1:\pi^2:0],\cdots,[1:\pi^{e-1}:0]$.
Again if $e=d$ and $p\not\equiv1~(\mathrm{mod}~e)$, then the point at infinity is only $[1:1:0]$.
Now if $e\neq d$ for $e<d$, the point at infinity is only $[0:1:0]$ and for $e>d$, then the point at infinity is only
$[1:0:0]$.
\end{rem}
\begin{rem}
 Let $q, e, d$ and $E_{e,d}$ be as in Theorems \ref{mt1} and \ref{mt2}. Let $a_q(E_{e,d}(\mathbb{F}_q))$ denotes
 the trace of Frobenius of the algebraic curve $E_{e,d}$. Since $a_q(E_{e,d}(\mathbb{F}_q))=q-N_{e,d}$,
 from Theorems \ref{mt1} and \ref{mt2}, we can express the trace of Frobenius of the algebraic curve $E_{e,d}$
 in terms of $_dF_{d-1}$ and $_{d-1}F_{d-2}$  Gaussian hypergeometric series containing multiplicative characters
 of orders $d$, $e(d-1)$ and $ed(d-1)$.
\end{rem}
\section{Applications}
In the following example, we deduce Theorem 2.1 of Lennon \cite{lennon2011} from Theorem \ref{mt2}.
\begin{exa}\cite{lennon2011}\label{ex1}
Let $q=p^n,p>3$ a prime and $q\equiv1(\mathrm{mod~12})$, Let $E_{2,3}:y^2=x^3+ax+b$ be an elliptic curve over $\mathbb{F}_q$
with $j(E_{2,3})\neq0,1728$, then the trace of Frobenius map on $E_{2,3}$ can be expressed as\\
\begin{align*}
a_q(E_{2,3})=-qT^{\frac{q-1}{4}}\left(\frac{a^3}{27}\right)
{_2}F_1
\left(
\begin{matrix}
T^{\frac{q-1}{12}},&T^{\frac{5(q-1)}{12}}\\
&&&|-\frac{27b^2}{4a^3}\\
&T^{\frac{q-1}{2}}
\end{matrix}\right).
\end{align*}
Let $e=2$ and $d=3$ in Theorem \ref{mt2}.
In this case, we have $M_1=qT^{-\frac{q-1}{12}}(-1)$, $N_1=q^2T^{-\frac{q-1}{12}}(-1)$ and\\
\begin{align*}
N_{2,3}=&q+\phi(b)-T^{-\frac{q-1}{12}}(-1)\phi(-2b)T^{\frac{q-1}{3}}(-1)+qT^{\frac{11(q-1)}{6}}(-1)\phi(2b)T^{-\frac{q-1}{12}}(-1) T^{\frac{q-1}{4}}\left(-\frac{27b^2}{4a^3}\right)\\
&\times{_2}F_1\left(\begin{matrix}
\eta,&\eta^5\\
&&\mid-\frac{27b^2}{4a^3}\\
&\psi^2\\
\end{matrix}
\right),
\end{align*}
where $\eta$ is a multiplicative character of order $12$ and $\psi$ is multiplicative character of order $4$. Thus, we have\\
\begin{align*}
N_{2,3}=q+\phi(b)-\phi(-2b)T^{\frac{q-1}{4}}(-1)+qT^{\frac{q-1}{4}}\left(\frac{a^3}{27}\right)
{_2}F_1\left(
\begin{matrix}
T^{\frac{q-1}{12}},&T^{\frac{5(q-1)}{12}}\\
&&\mid-\frac{27b^2}{4a^3}\\
&\phi\\
\end{matrix}
\right).
\end{align*}
Since $q\equiv1\pmod{12}$, therefore, $\phi(2)=T^{\frac{q-1}{4}}(-1)$ and $\phi(-1)=1$. Thus, $\phi(2b)T^{\frac{q-1}{4}}(-1)=\phi(b)$.
Hence, we have
\begin{align*}
N_{2,3}=q+qT^{\frac{q-1}{4}}\left(\frac{a^3}{27}\right)
{_2}F_1\left(
\begin{matrix}
T^{\frac{q-1}{12}},&T^{\frac{5(q-1)}{12}}\\
&&\mid-\frac{27b^2}{4a^3}\\
&\phi\\
\end{matrix}
\right).
\end{align*}
Since $a_q(E_{2,3})=q-N_{2,3}$, we have
\begin{align*}
a_q(E_{2,3})=-qT^{\frac{q-1}{4}}\left(\frac{a^3}{27}
\right)
{_2}F_1\left(
\begin{matrix}
T^{\frac{q-1}{12}},&T^{\frac{5(q-1)}{12}}\\
&&\mid-\frac{27b^2}{4a^3}\\
&\phi\\
\end{matrix}
\right).
\end{align*}
\end{exa}
We now give an example to show how Theorem \ref{mt2} is applied for specific values of $e$ and $d$.
\begin{exa}
If $q\equiv1(\mathrm{mod~36})$ and $E_{3,4}$ is an algebraic curve over $\mathbb{F}_q$, then
the trace of Frobenius map on $E_{3,4}$ can be expressed as\\
\begin{align*}
a_{q}(E_{3,4}(\mathbb{F}_q))=&-T^{-\frac{q-1}{3}}(b)-T^{-\frac{2(q-1)}{3}}(b)-q^3\binom{T^{\frac{4(q-1)}{9}}}{T^{\frac{q-1}{3}}}
\binom{T^{\frac{q-1}{36}}}{T^{\frac{5(q-1)}{36}}}T^{\frac{q-1}{3}}\left(\frac{3}{b}\right)\\
&\times{_4}F_3
\left(
\begin{matrix}
\phi,&\varepsilon,&T^{\frac{q-1}{4}},&T^{\frac{3(q-1)}{4}}\\
&&&&\mid-\frac{256b^3}{27a^4}\\
&T^{\frac{5(q-1)}{9}},&T^{\frac{2(q-1)}{9}},&T^{\frac{8(q-1)}{9}}\\
\end{matrix}
\right)-q^3\binom{T^{\frac{5(q-1)}{9}}}{T^{\frac{2(q-1)}{3}}}
\binom{T^{\frac{5(q-1)}{36}}}{T^{\frac{q-1}{36}}}\\
&\times T^{\frac{2(q-1)}{9}}(-1)T^{\frac{2(q-1)}{3}}\left(\frac{3}{b}\right)
{_4}F_3
\left(
\begin{matrix}
\phi,&\varepsilon,&T^{\frac{q-1}{4}},&T^{\frac{3(q-1)}{4}}\\
&&&&\mid-\frac{256b^3}{27a^4}\\
&T^{\frac{4(q-1)}{9}},&T^{\frac{q-1}{9}},&T^{\frac{7(q-1)}{9}}\\
\end{matrix}
\right).
\end{align*}
Let $e=3$ and $d=4$ in Theorem \ref{mt1}. In this case, we have
\begin{align*}
N_{3,4}=q+\sum_{i=1}^{2}T^{-\frac{i(q-1)}{3}}(b)+T^{-\frac{7(q-1)}{12}}(-1)\sum_{i=1}^{2}M_iT^{\frac{i(q-1)}{3}}
\left(\frac{3}{b}\right)
\end{align*}
\begin{align}
\times{_4}F_3
\left(
\begin{matrix}
\phi,&\varepsilon,&\chi,&\chi^3\\
&&&&\mid-\frac{256b^3}{27a^4}\\
&\psi^{6-i},&\psi^{3-i},&\psi^{9-i}\\
\end{matrix}
\right),\label{e2}
\end{align}
where $\chi$ and $\psi$ are characters of orders $4$ and $9$ respectively, and
\begin{align*}
M_i=\left\{ G_{-\frac{i(q-1)}{3}}G_{-\frac{(6-i)(q-1)}{9}}G_{\frac{(4i-3)(q-1)}{36}}G_{\frac{(4i-9)(q-1)}{36}}\right\},
i=1,2.
\end{align*}
Using Lemmas \ref{l2.3} and \ref{l2.4}, we find the values of $M_i$ for $i=1$,
\begin{align*}
M_1=&q^3\binom{T^{\frac{4(q-1)}{9}}}{T^{\frac{q-1}{3}}}\binom{T^{\frac{q-1}{36}}}{T^{\frac{5(q-1)}{36}}}
T^{\frac{7(q-1)}{12}}(-1),
\end{align*}
and for $i=2$,
\begin{align*}
M_2=&q^3\binom{T^{\frac{5(q-1)}{9}}}{T^{\frac{2(q-1)}{3}}}\binom{T^{\frac{5(q-1)}{36}}}{T^{\frac{q-1}{36}}}
T^{\frac{29(q-1)}{36}}(-1).\\
\end{align*}
By putting the value of $M_i$ in equation \eqref{e2}, we have
\begin{align*}
N_{3,4}=&q+T^{-\frac{q-1}{3}}(b)+T^{-\frac{2(q-1)}{3}}(b)+q^3\binom{T^{\frac{4(q-1)}{9}}}{T^{\frac{q-1}{3}}}
\binom{T^{\frac{q-1}{36}}}{T^{\frac{5(q-1)}{36}}}T^{\frac{q-1}{3}}\left(\frac{3}{b}\right)\\
&\times{_4}F_3
\left(
\begin{matrix}
\phi,&\varepsilon,&T^{\frac{q-1}{4}},&T^{\frac{3(q-1)}{4}}\\
&&&&\mid-\frac{256b^3}{27a^4}\\
&T^{\frac{5(q-1)}{9}},&T^{\frac{2(q-1)}{9}},&T^{\frac{8(q-1)}{9}}\\
\end{matrix}
\right)+q^3\binom{T^{\frac{5(q-1)}{9}}}{T^{\frac{2(q-1)}{3}}}
\binom{T^{\frac{5(q-1)}{36}}}{T^{\frac{q-1}{36}}}\\
&\times T^{\frac{2(q-1)}{9}}(-1)T^{\frac{2(q-1)}{3}}\left(\frac{3}{b}\right)
{_4}F_3
\left(
\begin{matrix}
\phi,&\varepsilon,&T^{\frac{q-1}{4}},&T^{\frac{3(q-1)}{4}}\\
&&&&\mid-\frac{256b^3}{27a^4}\\
&T^{\frac{4(q-1)}{9}},&T^{\frac{q-1}{9}},&T^{\frac{7(q-1)}{9}}\\
\end{matrix}
\right).
\end{align*}
Since $a_q(E_{3,4})=q-N_{3,4}$, thus
\begin{align*}
a_{q}(E_{3,4}(\mathbb{F}_q))=&-T^{-\frac{q-1}{3}}(b)-T^{-\frac{2(q-1)}{3}}(b)-q^3\binom{T^{\frac{4(q-1)}{9}}}{T^{\frac{q-1}{3}}}
\binom{T^{\frac{q-1}{36}}}{T^{\frac{5(q-1)}{36}}}T^{\frac{q-1}{3}}\left(\frac{3}{b}\right)\\
&\times{_4}F_3
\left(
\begin{matrix}
\phi,&\varepsilon,&T^{\frac{q-1}{4}},&T^{\frac{3(q-1)}{4}}\\
&&&&\mid-\frac{256b^3}{27a^4}\\
&T^{\frac{5(q-1)}{9}},&T^{\frac{2(q-1)}{9}},&T^{\frac{8(q-1)}{9}}\\
\end{matrix}
\right)-q^3\binom{T^{\frac{5(q-1)}{9}}}{T^{\frac{2(q-1)}{3}}}
\binom{T^{\frac{5(q-1)}{36}}}{T^{\frac{q-1}{36}}}\\
&\times T^{\frac{2(q-1)}{9}}(-1)T^{\frac{2(q-1)}{3}}\left(\frac{3}{b}\right)
{_4}F_3
\left(
\begin{matrix}
\phi,&\varepsilon,&T^{\frac{q-1}{4}},&T^{\frac{3(q-1)}{4}}\\
&&&&\mid-\frac{256b^3}{27a^4}\\
&T^{\frac{4(q-1)}{9}},&T^{\frac{q-1}{9}},&T^{\frac{7(q-1)}{9}}\\
\end{matrix}
\right).
\end{align*}
\end{exa}
We recall the following theorems from \cite{barman2013gautam,greene1987hypergeometric}.
\begin{thm}\cite{barman2013gautam}\label{t4.1}
Let $q=p^n$, $p>0$ an odd prime and let $T$ be a generator of the character group $\mathbb{\widehat F}_q^*$. The number of points
on the twisted Edward curve $C_{\alpha,\beta}:\alpha x^2+y^2=1+\beta x^2y^2$ over $\mathbb{F}_q$ can be expressed as
\begin{align*}
\#C_{\alpha,\beta}(\mathbb{F}_q)=q-1-\phi(\beta)-\phi(\alpha\beta)+q\phi(-\alpha)
{_2}F_1
\left(
\begin{matrix}
\phi,&\phi\\
&&\mid\frac{\beta}{\alpha}\\
&\varepsilon\\
\end{matrix}
\right).
\end{align*}
\end{thm}
\begin{thm}\cite{greene1987hypergeometric}
 If $A,B$ and $C$ are characters of $\mathbb{F}_q$, then
\begin{align*} 
 {_2}F_1
\left(
\begin{matrix}
A,&\overline{A}\\
&&\mid\frac{1}{2}\\
&\overline{A}B\\
\end{matrix}
\right)=A(-2)
\begin{cases}
0\quad\quad\quad\quad\quad\mbox{if}~B~\mbox{is not square}\\
\binom{C}{A}+\binom{\phi C}{A} \quad\mbox{if}~B=C^2.\\
\end{cases}
\end{align*}
\end{thm}
We have the following special case of the above theorem.
\begin{corollary}\label{c4.1}
 If $A=\phi$, $B=\phi$ and $q\equiv1 (\mathrm{mod}~4)$, then we have
 \begin{align*} 
 {_2}F_1
\left(
\begin{matrix}
\phi,&\phi\\
&&\mid\frac{1}{2}\\
&\varepsilon\\
\end{matrix}
\right)=\phi(-2)\left[\binom{T^{\frac{q-1}{4}}}{\phi}+\binom{T^{\frac{3(q-1)}{4}}}{\phi}\right].
 \end{align*}
\end{corollary}
\begin{thm}\label{t4.2}
Let $q=p^n$, $p>0$ an odd prime and $q\equiv1 (\mathrm{mod}~12)$.
If $a,b,k\in\mathbb{F}_q^*$ and $3k+a=0$, then the number of points on the elliptic curve 
$E_{a,b,0}:y^2=x^3+ax^2+bx$ can be expressed as 
\begin{align*}
\#E_{a,b,0}=q+qT^{\frac{3(q-1)}{4}}\left(\frac{3k^2+2ak+b}{3}\right)
{_2}F_1
\left(
\begin{matrix}
T^{\frac{q-1}{12}},&T^{\frac{5(q-1)}{12}}\\
&&\mid-\frac{27(k^3+ak^2+bk)^2}{4(3k^2+2ak+b)^3}\\
&\phi\\
\end{matrix}
\right).
\end{align*}
\end{thm}
\begin{proof}
Since $a\neq0$, we find $k\in\mathbb{F}_q^*$ such that $3k+a=0$. A change of variables $(x,y)\mapsto(x+k,y)$ takes 
the algebraic curve $E_{a,b,0}$ to birationally equivalent form $E^\prime_{2,3}:y^2=x^3+(3k^2+2ak+b)x+(k^3+ak^2+bk)$.
Clearly $\#E_{a,b,0}=\#E^\prime_{2,3}$, using Example \ref{ex1} for $a^\prime=3k^2+2ak+b$ and $b^\prime=k^3+ak^2+bk$, we have
\begin{align*}
 \#E_{a,b,0}=q+qT^{\frac{3(q-1)}{4}}\left(\frac{3k^2+2ak+b}{3}\right)
{_2}F_1
\left(
\begin{matrix}
T^{\frac{q-1}{12}},&T^{\frac{5(q-1)}{12}}\\
&&\mid-\frac{27(k^3+ak^2+bk)^2}{4(3k^2+2ak+b)^3}\\
&\phi\\
\end{matrix}
\right).
\end{align*}
\end{proof}
\begin{corollary}\label{c4.2}
Let $q=p^n$, $p>0$ an odd prime and $q\equiv1(\mathrm{mod}~12)$. If $a,b,k\in\mathbb{F}_q^*$, $b$ is a square, $3k+a=0$ and
$a\neq\pm2\sqrt{b}$, then
\begin{align*}
&qT^{\frac{3(q-1)}{4}}\left(\frac{3k^2+2ak+b}{3}\right){_2}F_1
\left(
\begin{matrix}
T^{\frac{q-1}{12}},&T^{\frac{5(q-1)}{12}}\\
&&\mid-\frac{27(k^3+ak^2+bk)^2}{4(3k^2+2ak+b)^3}\\
&\phi\\
\end{matrix}
\right)=\\
&-\phi(a-2\sqrt{b})
+\phi(ab-2b\sqrt{b})+q\phi(-a-2\sqrt{b})
{_2}F_1
\left(
\begin{matrix}
\phi,&\phi\\
&&\mid\frac{a-2\sqrt{b}}{a+2\sqrt{b}}\\
&\varepsilon\\
\end{matrix}
\right).
\end{align*}
\end{corollary}
\begin{proof}
A change of variables $(x,y)\mapsto\left(\frac{x}{y},\frac{x-\sqrt{b}}{x+\sqrt{b}}\right)$ takes the algebraic curve $E_{a,b,0}:y^2=x^3+ax^2+bx$ 
to birationally equivalent form  $C_{\alpha,\beta}:\alpha x^2+y^2=1+\beta x^2y^2$, where $\alpha=a+2\sqrt{b}$
 and $\beta=a-2\sqrt{b}$.
Now the points on $E_{a,b,0}$ for $y=0$ and $x=-\sqrt{b}$ do not correspond to any points on $C_{\alpha,\beta}$.
For $y=0$,  there are $2+T^{\frac{q-1}{2}}(a^2-4b)$ extra points on $E_{a,b,0}$, and 
for $x=-\sqrt{b}$, there are $1+T^{\frac{q-1}{2}}(ab-2b\sqrt{b})$ extra points on $E_{a,b,0}$.
Similarly under the inverse transformation, a change of variables $(x,y)\mapsto
\left(\frac{\sqrt{b}(1+y)}{1-y},\frac{\sqrt{b}(1+y)}{x(1-y)}\right)$ takes the algebraic curve $C_{\alpha,\beta}$ to birationally equivalent form $E_{a,b,0}$. 
Now the points on $C_{\alpha,\beta}$ for $x=0$ and $y=1$ do not correspond to any points on $E_{a,b,0}$, so
for $x=0$ and $y=1$, there are two extra point $(0,1),(0,-1)$ on $C_{\alpha,\beta}$.
Clearly $\#E_{a,b,0}+2=\#C_{\alpha,\beta}+3+T^{\frac{q-1}{2}}(a^2-4b)+T^{\frac{q-1}{2}}(ab-2b\sqrt{b})$.
From Theorems \ref{t4.1} and \ref{t4.2}, we have
\begin{align*}
&q+qT^{\frac{3(q-1)}{4}}\left(\frac{3k^2+2ak+b}{3}\right){_2}F_1
\left(
\begin{matrix}
T^{\frac{q-1}{12}},&T^{\frac{5(q-1)}{12}}\\
&&\mid-\frac{27(k^3+ak^2+bk)^2}{4(3k^2+2ak+b)^3}\\
&\phi\\
\end{matrix}
\right)+2=\\
&q-1-\phi(a-2\sqrt{b})
-\phi(a^2-4b)+3+\phi(a^2-4b)+\phi(ab-2b\sqrt{b})\\
&+q\phi(-a-2\sqrt{b})
{_2}F_1
\left(
\begin{matrix}
\phi,&\phi\\
&&\mid\frac{a-2\sqrt{b}}{a+2\sqrt{b}}\\
&\varepsilon\\
\end{matrix}
\right).
\end{align*}
By solving the above equation, we have
\begin{align*}
 &qT^{\frac{3(q-1)}{4}}\left(\frac{3k^2+2ak+b}{3}\right){_2}F_1
\left(
\begin{matrix}
T^{\frac{q-1}{12}},&T^{\frac{5(q-1)}{12}}\\
&&\mid-\frac{27(k^3+ak^2+bk)^2}{4(3k^2+2ak+b)^3}\\
&\phi\\
\end{matrix}
\right)=\\
&-\phi(a-2\sqrt{b})+\phi(ab-2b\sqrt{b})+q\phi(-a-2\sqrt{b})
{_2}F_1
\left(
\begin{matrix}
\phi,&\phi\\
&&\mid\frac{a-2\sqrt{b}}{a+2\sqrt{b}}\\
&\varepsilon\\
\end{matrix}
\right).
\end{align*}
\end{proof}
In the next theorem, we obtain some special values of $_2F_1$ hypergeometric function containing characters of order $12$.
\begin{thm}
Let $q=p^n$, $p$ be a prime with $q\equiv1(\mathrm{mod}~12)$ and let $T$ be a fixed generator of $\mathbb{F}_q^*$, then
\begin{enumerate}
 \item
$
{_2}F_1
\left(
\begin{matrix}
T^{\frac{q-1}{12}},&T^{\frac{5(q-1)}{12}}\\
&&\mid\frac{1323}{1331}\\
&\phi\\
\end{matrix}
\right)=T^{\frac{q-1}{4}}\left(\frac{-44}{3}\right)\phi(2)\left[\binom{T^{\frac{q-1}{4}}}{\phi}+\binom{T^{\frac{3(q-1)}{4}}}{\phi}
\right],
$

\item
$
{_2}F_1
\left(
\begin{matrix}
T^{\frac{q-1}{12}},&T^{\frac{5(q-1)}{12}}\\
&&\mid\frac{8}{1331}\\
&\varepsilon\\
\end{matrix}
\right)=T^{\frac{q-1}{3}}(-1)T^{\frac{q-1}{4}}\left(\frac{44}{3}\right)\phi(2)
\left[\binom{T^{\frac{q-1}{4}}}{\phi}+\binom{T^{\frac{3(q-1)}{4}}}{\phi}
\right],
$
\item
$
{_2}F_1
\left(
\begin{matrix}
T^{\frac{q-1}{12}},&T^{\frac{q-1}{12}}\\
&&\mid\frac{-1323}{8}\\
&\phi\\
\end{matrix}
\right)=T^{\frac{q-1}{12}}\left(\frac{8}{1331}\right)T^{\frac{q-1}{4}}\left(-\frac{44}{3}\right)\phi(2)
\left[\binom{T^{\frac{q-1}{4}}}{\phi}+\binom{T^{\frac{3(q-1)}{4}}}{\phi}
\right],
$
\item
$
{_2}F_1
\left(
\begin{matrix}
T^{\frac{q-1}{12}},&T^{\frac{-5(q-1)}{12}}\\
&&\mid\frac{-8}{1323}\\
&\varepsilon\\
\end{matrix}
\right)=T^{\frac{q-1}{3}}(-1)T^{\frac{q-1}{12}}\left(\frac{1323}{1331}\right)T^{\frac{q-1}{4}}\left(\frac{44}{3}\right)\phi(2)
\left[\binom{T^{\frac{q-1}{4}}}{\phi}+\binom{T^{\frac{3(q-1)}{4}}}{\phi}\right].
$
\end{enumerate}
\end{thm}
\begin{proof}
Set $a=12$ and $b=4$ in Corollary \ref{c4.2}, we have
\begin{align*}
qT^{\frac{3(q-1)}{4}}\left(\frac{3k^2+24k+4}{3}\right)
{_2}F_1
\left(
\begin{matrix}
T^{\frac{q-1}{12}},&T^{\frac{5(q-1)}{12}}\\
&&\mid-\frac{27(k^3+12k^2+4k)^2}{4(3k^2+24k+4)^3}\\
&\phi\\
\end{matrix}
\right)\\
=-\phi(8)+\phi(32)+q\phi(-16)
{_2}F_1
\left(
\begin{matrix}
\phi,&\phi\\
&&\mid\frac{1}{2}\\
&\varepsilon\\
\end{matrix}
\right).
\end{align*}
Since in Corollary \ref{c4.2}, $3k+a=0$, so $k$ becomes $-4$, then we have
\begin{align*}
{_2}F_1
\left(
\begin{matrix}
T^{\frac{q-1}{12}},&T^{\frac{5(q-1)}{12}}\\
&&\mid\frac{1323}{1331}\\
&\phi\\
\end{matrix}
\right)
=\phi(-1)T^{\frac{q-1}{4}}\left(\frac{-44}{3}\right)
{_2}F_1
\left(
\begin{matrix}
\phi,&\phi\\
&&\mid\frac{1}{2}\\
&\varepsilon\\
\end{matrix}
\right).
\end{align*}
Using Corollary \ref{c4.1}, the above equation becomes\\
\begin{align*}
{_2}F_1
\left(
\begin{matrix}
T^{\frac{q-1}{12}},&T^{\frac{5(q-1)}{12}}\\
&&\mid\frac{1323}{1331}\\
&\phi\\
\end{matrix} 
\right)
=\phi(2)T^{\frac{q-1}{4}}\left(\frac{-44}{3}\right)
\left[\binom{T^{\frac{q-1}{4}}}{\phi}+
\binom{T^\frac{3(q-1)}{4}}{\phi}\right].
\end{align*}
$(2)$
Putting $x=\frac{1323}{1331}$ $A=T^{\frac{q-1}{12}}, B=T^{\frac{5(q-1)}{12}}$ and $C=\phi$ in Theorem \ref{tg4.4}-$(1)$, we obtain
\begin{align*}
 {_2}F_1
\left(
\begin{matrix}
T^{\frac{q-1}{12}},&T^{\frac{5(q-1)}{12}}\\
&&\mid\frac{8}{1331}\\
&\varepsilon\\
\end{matrix} 
\right)
=T^{\frac{q-1}{12}}(-1)
{_2}F_1
\left(
\begin{matrix}
T^{\frac{q-1}{12}},&T^{\frac{5(q-1)}{12}}\\
&&\mid\frac{1323}{1331}\\
&\phi\\
\end{matrix} 
\right).
\end{align*}
Thus the proof of $(2)$ follows from the proof of $(1)$.\\
$(3)$ For $x=\frac{1323}{1331}$, $A=T^{\frac{q-1}{12}},~B=T^{\frac{5(q-1)}{12}}$ and $C=\phi$, Theorem \ref{tg4.4}-$(2)$
yields
\begin{align*}
 {_2}F_1
\left(
\begin{matrix}
T^{\frac{q-1}{12}},&T^{\frac{q-1}{12}}\\
&&\mid-\frac{1323}{8}\\
&\phi\\
\end{matrix} 
\right)
=T^{\frac{q-1}{12}}\left(\frac{8}{1331}\right)\phi(-1)
{_2}F_1
\left(
\begin{matrix}
T^{\frac{q-1}{12}},&T^{\frac{5(q-1)}{12}}\\
&&\mid\frac{1323}{1331}\\
&\phi\\
\end{matrix} 
\right).
\end{align*}
Since $q\equiv 1\pmod{12}$, therefore $\phi(-1)=1$,
then the proof of $(3)$ follows from the proof of $(1)$.\\
$(4)$ Finally putting $x=\frac{8}{1331}$, $A=T^{\frac{q-1}{12}},~B=T^{\frac{5(q-1)}{12}}$ and $C=\varepsilon$ in
Theorem \ref{tg4.4}-$(2)$, we obtain
\begin{align*}
 {_2}F_1
\left(
\begin{matrix}
T^{\frac{q-1}{12}},&T^{\frac{-5(q-1)}{12}}\\
&&\mid-\frac{8}{1323}\\
&\varepsilon\\
\end{matrix} 
\right)
=T^{\frac{q-1}{12}}\left(\frac{1323}{1331}\right)
{_2}F_1
\left(
\begin{matrix}
T^{\frac{q-1}{12}},&T^{\frac{5(q-1)}{12}}\\
&&\mid\frac{8}{1331}\\
&\varepsilon\\
\end{matrix} 
\right).
\end{align*}
Hence the proof of $(4)$ follows from the proof of $(2)$.
\end{proof}
\begin{center}
 \textbf{Acknowledgment}
\end{center}
We would like to thank Rupam Barman for a careful reading of the initial drafts of this paper and suggesting many modification which 
have improved the exposition.
\bibliographystyle{plain} 
\bibliography{ref}
\end{document}